\documentclass{amsart}

\usepackage{amsmath,amsthm,amscd,amssymb}
\usepackage{mathrsfs}
\usepackage{graphicx}
\usepackage{enumerate}

\usepackage{ascmac}
\usepackage{color}
\usepackage{here}

\theoremstyle{plain}
\newtheorem{thm}{Theorem}[section]

\newtheorem{lem}[thm]{Lemma}
\newtheorem{prop}[thm]{Proposition}

\theoremstyle{definition}
\newtheorem{dfn}[thm]{Definition}

\theoremstyle{remark}
\newtheorem{remark}[thm]{Remark}

\numberwithin{equation}{section}

\title[Highly irregular orbits for subshifts of finite type]{Highly irregular orbits for subshifts of finite type: large intersections and emergence}

\date{\today}

\author{Yushi Nakano}
\address[Yushi Nakano]{Department of Mathematics, Tokai University, 4-1-1 Kitakaname, Hiratuka, Kanagawa, 259-1292, JAPAN}
\email{yushi.nakano@tsc.u-tokai.ac.jp}

\author{Agnieszka Zelerowicz}
\address[Agnieszka Zelerowicz]{Department of Mathematics, University of Maryland, College Park, MD 20742, USA}
\email{azelerow@umd.edu}

\begin{document}

\begin{abstract}
In their recent paper \cite{KNS2019}, the first author,  S. Kiriki, and T. Soma introduced a concept of pointwise emergence to %quantitatively study non-existence of averages, and 
measure the complexity of irregular orbits.
They
constructed a residual subset of the full shift with high pointwise emergence.
In this paper %we show that for any topologically mixing subshift of finite type, 
%there exists a residual subset of the state space with high pointwise emergence, full topological entropy, full Hausdorff dimension, and full topological pressure for any H\"older continuous potential.
%Furthermore, we show that this set belongs to a certain class of sets with large intersection property.
we consider the set of points with high pointwise
emergence for topologically mixing subshifts of finite type. We show that this set has full topological
entropy, full Hausdorff dimension, and full topological pressure for any H\"older
continuous potential. 
Furthermore, we show that this set belongs to a certain
class of sets with large intersection property.
This is a natural generalization of \cite{FP2011} to
pointwise emergence and Carath\'eodory dimension.
\end{abstract}

%Abstract of Agnieszka's talk in Banylus
%Recently, P. Berger introduced a concept of metric emergence to
%quantitatively study the non-uniqueness of statistics (such as Newhouse or
%KAM phenomena). Later, S. Kiriki, Y. Nakano, and T. Soma introduced a
%concept of pointwise emergence to quantitatively study non-existence of aver-
%ages, and constructed a residual subset of the full shift with high pointwise
%emergence. In this talk I will consider the set of points with high pointwise
%emergence for topologically mixing subshifts of finite type. I will present my
%joint work with Yushi Nakano, where we show that this set has full topological
%entropy, full Hausdorff dimension, and full topological pressure for any H\CID{124}lder
%continuous potential. Furthermore, we show that this set belongs to a certain
%class of sets with large intersection property.

\maketitle

\section{Introduction}\label{section:introduction}

Let $X$ be a compact metric space and $f: X\to X$ a continuous map.
Let  $\mathcal P(X)$ be the set of Borel probability measures on $X$ equipped with the weak topology.
A point $x\in X$
  is said to be \emph{irregular}, if the time average along the forward orbit of $x$ does not exist, i.e.~if
 the limit of empirical measures, 
 \begin{equation}\label{eqn:bir-average}
\delta _x^n = \frac{1}{n} \sum _{j=0}^{n-1} \delta _{f^j (x)}, \quad n\geq 1,
\end{equation}
 does not exist in $\mathcal P(X)$
 (see \cite{Thompson2010,ABC2011}).
 Such a point is also called  \emph{historic} \cite{Ruelle2001, Takens2008}, 
 \emph{non-typical}  
 \cite{BS2000} 
  or \emph{divergent} 
   \cite{CKS2005}.

Although the set of irregular points  (which we will call the \emph{irregular set} and denote by $I$) is a $\mu$-zero measure set for any invariant measure $\mu$ due to Birkhoff's ergodic theorem, the set is known to be remarkably large for abundant dynamical systems.
Pesin and Pitskel' \cite{PP84}   obtained the first result for the largeness of the irregular set from thermodynamic  viewpoint.  
In their paper, they showed that    the irregular set for the full shift has full topological entropy and full Hausdorff dimension, that is, 
\[
 h_{\mathrm{top}}( I ) = h_{\mathrm{top}}(X) \quad \text{and} \quad \mathrm{dim} _H(I ) = \mathrm{dim} _H(X) .
\]
Here $h_{\mathrm{top}}(Z)$ is the Pesin-Pitskel' topological entropy for a (not necessarily compact) Borel set $Z$ given in \cite{PP84} (see Subsection \ref{subsection:pressure} for precise definition; this is a generalization of Bowen's Hausdorff topological entropy of  \cite{Bowen1973}, and we refer to  \cite{HNP2008} for relation between entropies for a non-compact set). 
This thermodynamic largeness of irregular sets was extended to topologically mixing subshifts of finite type in \cite{BS2000} (together with the detailed study of the set of points at which Lyapunov exponent or local entropy fail to exist), 
to graph directed Markov systems in \cite{FP2011},
to continuous maps with specification property in \cite{CKS2005} (see also \cite{Thompson2010}), 
and  to continuous maps with almost specification property
in \cite{Thompson2012}.
See also \cite{HK1990, Ruelle2001, Takens2008, ABC2011, CZZ2011, CTV2015, Tian2017, BV2017, KS2017, AP2019, BKNSR2020, Yang2020}
 and references therein for the study of irregular sets from other viewpoints.

The irregular points considered in the above works were typically those, whose corresponding sequence of measures (\ref{eqn:bir-average}) oscillates between two (or finitely many) ergodic measures.
Consequently, the space of accumulation points of (\ref{eqn:bir-average}) is finitely dimensional.
In this paper we consider irregular points with high complexity, that is points for which (\ref{eqn:bir-average}) oscillates between infinitely many ergodic measures. More precisely, given an infinite sequence of distinct ergodic measures, we consider points for which the set of accumulation points of (\ref{eqn:bir-average}) contains the infinite-dimensional simplex spanned by those measures.
We obtain a lower bound on the dimension of the set of such points as the infimum over dimensions of measures in the sequence (See Theorem \ref{cor:main}).
    
Recently, Berger introduced a concept of metric emergence in \cite{Berger2017} 
in order to 'evaluate the complexity to
approximate a system by statistics' \cite{Berger2017}.
Metric emergence quantifies such phenomenas as the Newhouse phenomenon or KAM phenomenon. Even though the notion of emergence is relatively new, the study of metric emergence is already becoming  a quite active research area \cite{Berger2020, BBo, BBi, Talebi, CJZ}.
Inspired by Berger's  work, the first author, S. Kiriki, and T. Soma \cite{KNS2019} introduced a concept of pointwise emergence to measure complexity of irregular orbits. 
    The \emph{pointwise emergence} $\mathscr E_x(\epsilon ) $  at scale $\epsilon >0$ at $x\in X$ is defined by
\begin{multline}\label{eq:0403d}
\mathscr E_x(\epsilon ) = \min \Big\{ N \in \mathbb N \mid  \text{there exists   $\{ \mu _j\} _{j=1}^N \subset \mathcal P(X)$ such that} \\
\limsup _{n\to \infty}  \min _{1\leq j\leq N}  d \left( \delta _x^n , \mu _j \right) \leq \epsilon \Big \},
\end{multline}
where $d$ is the first Wasserstein metric\footnote{
For $j=1,2$, let $p_j :X\times X\to X $ be the canonical projection to the $j$-th coordinate, and $(p_j) _*\pi$ the pushforward measure of a probability measure $\pi$ on $X\times X$ by $p_j$. 
 Let $\Pi (\mu, \nu )$ be the set of probability measures $\pi$ on $X\times X$ such that $(p_1)_*\pi =\mu$ and $(p_2)_*\pi =\nu$. 
 The first Wasserstein metric $d $ is defined as
 \[
d(\mu ,\nu )= \inf _{\pi \in \Pi (\mu , \nu )} \int _{X\times X} d_X(x,y) d\pi (x,y) \quad \text{for $ \mu  ,\nu \in \mathcal P(X)$} .
 \]
} 
on $\mathcal P(X)$ (see \cite{Villani2003, Villani2008}  for its  properties; we here merely recall that $d$ is a metrization of the weak topology of $\mathcal P(X)$).
According to \cite[Proposition 1.2]{KNS2019}, $x$ is irregular
if and only if
\[
\lim _{\epsilon \to 0}  \mathscr E_x(\epsilon ) =\infty .
\]
When $\min _{j}  d \left( \delta _x^n , \mu _j \right)$ in \eqref{eq:0403d} is replaced by $\int _X \min _{j}  d \left( \delta _x^n , \mu _j \right) m (dx)$ with some $m\in \mathcal P(X)$, then the quantity given by  \eqref{eq:0403d} is called the metric emergence with respect to $m$.
%For the relationship between metric and pointwise emergences, we refer to \cite{KNS2019}.
A fundamental relationship between metric and pointwise emergences is the following inequality:
\begin{equation}\label{eq:0901}
\min _{x\in D}  \mathscr E_x(\epsilon ) \leq  \mathscr E_m(m(D) \, \epsilon )
\end{equation}
for every $\epsilon >0$, Borel set $D$ and $m\in \mathcal P(X)$, see  \cite[Proposition 1.4]{KNS2019}.
The pointwise emergence at $x\in X$ is called \emph{super-polynomial} (or \emph{high})
 if 
\[
\limsup _{\epsilon \to 0}  \frac{\log \mathscr E_x(\epsilon ) }{-\log \epsilon } =\infty .
\] 
As is pointed out in \cite{Berger2017}, it  is widely accepted among computer scientists that super-polynomial algorithms are impractical. From that perspective dynamical systems with high metric emergence are not feasible to be studied numerically.
The set of points with high (pointwise) emergence can be considered as statistically very complex (see also \cite{Berger2020,BBo} for other motivations to study high emergence).
Therefore, it is of great interest to investigate how large the set of points with high emergence is.

The main result of this paper is as follows.
\begin{thm}\label{thm:main}
Let $X $ be a topologically mixing subshift of finite type of $  \{ 1, 2, \ldots , \kappa\} ^{\mathbb N}$ with $\kappa\geq 2$.\footnote{We endow it with a standard metric 
\[
d_{X}(x ,y ) = \sum _{j=1}^\infty \frac{\vert x_j -y_j \vert }{\beta ^j} \quad \text{for $x=(x_1 ,x_2, \ldots ), y=(y_1, y_2, \ldots ) \in X$}
\]
 with some $\beta >1$.}
 Let $f:X \to X$ be the left-shift operation on $X$.  
 Let $E$ be the set of points $x\in X$ satisfying
 \begin{equation*}
 \lim _{\epsilon \to 0}  \frac{\log \mathscr E_x(\epsilon ) }{-\log \epsilon } =\infty.
%  \quad \text{for all 
  % $x\in E$}.
\end{equation*}
Then, 
\[
h_{\mathrm{top}} (E) =h_{\mathrm{top}} (X)  \quad \text{and } \quad \mathrm{dim}_{H} (E) =\mathrm{dim}_{H} (X) .
\]
In addition, for any H\"older continuous function $\varphi$, we have that $P_{E}(\varphi)=P_X(\varphi)$.
That is, the set of points with high emergence carries full topological pressure.
\end{thm}

Notice that   $\mathrm{dim} _H(A) = h_{\mathrm{top}}(A) /\beta $ for any $A\subset X$ in the setting of Theorem \ref{thm:main}, 
  so $h_{\mathrm{top}} (E) =h_{\mathrm{top}} (X) $ is equivalent to $ \mathrm{dim}_{H} (E) =\mathrm{dim}_{H} (X) $.
Our approach to proving Theorem \ref{thm:main} is a generalization of ideas in \cite{FP2011}. In \cite{FP2011} the authors introduced classes of sests, $\mathcal{G}^s$, $0<s<  \mathrm{dim} _H(X),$ such that: (a) every countable intersection of sets in a given class $\mathcal{G}^s$ also belongs to $\mathcal{G}^s$;
(b) every set in $\mathcal{G}^s$ has Hausdorff dimension at least $s$.

This was later used to show that a certain subset of the set of irregular points has full Hausdorff dimension. We extend this result in two directions. We consider a general Carath\'eodory dimension structure (see Section \ref{car-struc}) introduced by Pesin in \cite{pes97}.
We then introduce classes of sets corresponding to this structure such that the analogue of $(a)$ and $(b)$ holds, under some conditions on the Carath\'eodory structure (see Section \ref{subsection:lic}). 
We then consider the set of points with high emergence and analyse when it belongs to such defined class of sets. 
As a result we obtain a more general version of Theorem \ref{thm:main}, which states that  
under certain conditions on the Carath\'eodory structure, the set of points with high emergence has full  Carath\'eodory dimension (see Theorem \ref{thm:main-general}). 

%As an immediate application of this more general Theorem \ref{thm:main-general} we obtain that for any H\"older continuous function $\varphi$, we have that $P_{E}(\varphi)=P_X(\varphi)$.
%That is, the set of points with high emergence carries full topological pressure.

\begin{remark}
It is possible to prove the first part of Theorem \ref{thm:main} (the statement about the Hausdorff dimension) by modifying the construction given by Barreira-Schmeling  in \cite{BS2000}. The advantage of our approach is that it gives a stronger result (see Theorem \ref{thm:main-general}).
Moreover we develop tools that could potentially be used to study largeness of other sets in addition to the set of points with high emergence (see Theorem \ref{thm:main1}).
\end{remark}

\begin{remark}
It immediately follows from (the statement for the Hausdorff dimension in)  Theorem \ref{thm:main} that $m^t_H (E) >0$ for any $t< \mathrm{dim }_H(X)$, where $E$ is the set of points with high pointwise emergence given in Theorem \ref{thm:main} and $m^t_H$ is the $t$-dimensional Hausdorff measure  (see Section \ref{car-struc} and \ref{s:app} for precise definition). 
By Frostman's lemma, this implies that there is a Borel probability measure $m$ such that the support of $m$ is included in $E$ (in particular, $m(E)>0$) and $m(B(x,r))\leq r^t$ for all $x\in X$ and $r>0$,
where $B(x,r)$ is the ball of  radius $r$ centered at  $x$.  
Therefore, by \eqref{eq:0901}, we have
\[
 \lim _{\epsilon \to 0}  \frac{\log \mathscr E_{m}(\epsilon ) }{-\log \epsilon } =\infty .
\]
We remark that Berger conjectured in \cite{Berger2017} that a ``typical''
  diffeomorphism $f$ on a compact manifold satisfies the above equation with  the normalized Lebesgue measure $\mathrm{Leb}$ of the   manifold in place of $m$, i.e.
  $ \lim _{\epsilon \to 0}  \log \mathscr E_{\mathrm{Leb}}(\epsilon ) /(-\log \epsilon ) =\infty $.
\end{remark}

\subsection{Structure of the paper}
In Section \ref{car-struc} we recall 
%the construction of 
Carath\'eodory dimension structure and introduce modifications to Carath\'eodory outer measures needed for our constructions. 
We also introduce conditions on the Carath\'eodory dimension structure needed for the results in this paper.
In Section \ref{subsection:lic} we introduce and study classes of sets with large intersection property.
We state the general version of our main result  (Theorem \ref{thm:main-general}), 
which is an immediate consequence of two results for large intersection classes (Theorems \ref{thm:main1} and \ref{cor:main}).
In Section \ref{s:app} we give examples of  Carath\'eodory dimension structures satisfying the assumptions of Theorem \ref{thm:main-general}.
As a Corollary, we obtain Theorem \ref{thm:main}.
In Section \ref{s:1}, we give the proof of Theorem \ref{thm:main1}. 
Section \ref{s:2} is dedicated to the proof of Theorem \ref{cor:main}.
%In Appendix \ref{a:1} we give a \emph{constructive} proof of Theorem \ref{thm:main}, although we have no idea whether the resulting set with high emergence, full topological dimension and full Hausdorff dimension is in the class of sets with the large intersection property.

\section{Carath\'eodory dimension structure}\label{car-struc}

We recall the construction introduced in \cite{pes97}, called the Carath\'eodory dimension structure. 

\subsubsection{Carath\'eodory dimension of sets and measures} Let $X$ be a set and $\mathcal{F}$ a collection of subsets of $X$ which we call \emph{admissible} sets. Assume that there exist two set functions 
$\eta,\psi :\mathcal{F} \to [0,\infty)$ satisfying the following conditions:
\begin{enumerate}
\item[(A1)] $\emptyset\in\mathcal{F}$; $\eta(\emptyset)=\psi(\emptyset)=0$ and $\eta(U),\psi(U)>0$ for any $U\in\mathcal{F}, U\ne\emptyset$;
\item[(A2)] for any $\delta>0$ one can find $\varepsilon>0$ such that 
$\eta(U)\le\delta$ for any $U\in\mathcal{F}$ with $\psi(U)\le\varepsilon$;
\item[(A3)] there exists a sequence of positive numbers $\epsilon_n \to 0$ such that for any 
$n\in\mathbb{N}$ one can find a finite subcollection 
$\mathcal{G}\subset\mathcal{F}$ covering $X$ such that 
$\psi(U)=\epsilon_n$ for any $U\in\mathcal{G}$.
\end{enumerate}
Let $\xi:\mathcal{F}\to[0,\infty)$ be a set function. We say that the collection of subsets $\mathcal{F}$ and the functions $\xi,\eta,\psi$, satisfying Conditions (A1), (A2) and (A3) introduce a \emph{Carath\'eodory dimension structure} or \emph{$C$-structure} $\tau$ on $X$ and we write 
$\tau=(\mathcal{F},\xi,\eta,\psi)$.

For any subcollection $\mathcal{G}\subset\mathcal{F}$ denote by
$\psi(\mathcal{G}):=\sup\{\psi(U) | U\in\mathcal{G}\}$. 
Given a set $Z\subset X$ and numbers $t\in\mathbb{R}$ and $\varepsilon>0$, define
$$ 
\mathfrak{M}^t_C(Z,\varepsilon):=\inf_{\mathcal{G},\psi(\mathcal{G})\le\varepsilon} \left\{\sum_{U\in\mathcal{G}}\xi(U)\eta(U)^{t}\right\},   
$$
where the infimum is taken over all finite or countable subcollections 
$\mathcal{G}\subset\mathcal{F}$ covering $Z$. 
Set
$$ 
m^t_C(Z):= \lim_{\varepsilon\to 0} \mathfrak{M}^t_C(Z,\varepsilon).  
$$
If $m^t_C(\emptyset)=0$, then the set function $m^t_C(\cdot)$ becomes an outer measure on $X$, which induces a measure on the 
$\sigma$-algebra of measurable sets. We call this measure the 
\emph{$t$-Carath\'eodory measure}. In general, this measure may not be $\sigma$-finite or it may be a zero measure. The following is shown in \cite{pes97}.

\begin{prop}\label{prop:crit-value}
For any set $Z\subset X$ there exists a critical value 
$t_C\in\mathbb{R}$ such that $m^t(Z)=\infty$ for $t<t_C$ and $m^t_C(Z)=0$ for 
$t>t_C$ (while $m^{t_C}_C(Z)$ may be $0$, $\infty$, or a finite positive number).
\end{prop}
We call $\dim_CZ =t_C$ the \emph{Carath\'eodory dimension} of the set $Z$. 
If $X$ is a measurable space with a measure $\mu$, then the quantity 
$$
\begin{aligned}
\dim_C\mu&=\inf\{\dim_CZ\colon \mu(Z)=1\}\\
&=\lim_{\delta\to 0}\inf\{\dim_CZ\colon \mu(Z)>1-\delta\}
\end{aligned}
$$
is called the \emph{Carath\'eodory dimension} of $\mu$.

\subsection{Modification of Carath\'eodory outer measures}\label{sec:modification}

For the rest of the paper we restrict our attention to $C$-structures $\tau=(\mathcal{F},\xi,\eta,\psi)$ on the shift space $X$.
Recall that $X\subset \{ 1,\ldots ,\kappa\}^{\mathbb N}$ is a subshift of finite type, meaning that there is a matrix $M =(M_{i,j})_{1\leq i,j \leq \kappa}$ such that each entry of $M$ is $0$ or $1$ and that $X$ consists of admissible words 
$x=(x_1 x_2\ldots )\in \{ 1,\ldots ,\kappa\}^{\mathbb N}$ with respect to $M$.
Here, $x=(x_1\cdots x_n)\in  \{ 1,\ldots ,\kappa\}^{n}$, $n\in \mathbb N\cup\{\infty \}$, is called admissible if $M_{x_j x_{j+1}} =1 $ for all $1\leq j< n$. 
We denote the length $n$ of the word  $x $ by $\vert x\vert $.
Furthermore, for a given admissible word $u=(u_1\ldots u_n)$, we define the cylinder  $C(u) $  by $C(u) =\{ x\in X \mid [x]_n =u\}$,
where
  $[x]_n$ is the truncation $[x]_n =(x_1x_2\ldots x_n)$ of $x=(x_1x_2\ldots )$.
Let the collection $\mathcal{F}$ of admissible sets be the collection of cylinders,
$$ 
\mathcal{F}:=\{\emptyset\}\cup\{C (u) \mid \text{$u$ is an admissible word}\}. 
$$
For a cylinder $C\in \mathcal{F}$  and $t\in\mathbb{R}$ denote 
$$ 
q(C,t):= \xi(C)\eta(C)^t  ,
$$
and  denote by $l(C)$ the smallest number $n$ such that $C=C(u)$ with some admissible word $u$ of length $n$ (also called the length of $C$).

For our purpose we suggest a modification of the Carath\'eodory outer measure similarly as it was done in \cite{FP2011} for Hausdorff outer measure.
The advantage of the new outer measure is that it is always finite.

Assume that a Carath\'eodory dimension structure $\tau=(\mathcal{F},\xi,\eta,\psi)$ on $X$ is given. For any set $Z\subset X$ and a number $t\in\mathbb{R}$ define 

$$ 
M^t_C(Z):=\inf_{\mathcal{G}} \left\{\sum_{U\in\mathcal{G}}\xi(U)\eta(U)^{t}\right\},   
$$
where the infimum is taken over all finite or countable sub-collections 
$\mathcal{G}\subset\mathcal{F}$ covering $Z$.

In Section \ref{s:2} we will consider yet another collection of outer measures, $N^{t}_{C,m}(\cdot)$, depending on the parameter $m\in\mathbb{N}$, defined identically as $M^t_C(\cdot)$, but with the family of admissible sets restricted to
$$ 
\mathcal{F}_m:=\{\emptyset\}\cup\{C(u) \mid \text{$u$ is an admissible word}, \;  \text{ $\vert u\vert = k \cdot m$  for some $ k\in\mathbb{N}$}\}\subset \mathcal{F} .
$$
 
In our arguments we consider an outer measure $N^{t}_{C,m}(\cdot)$ with a fixed large parameter $m\geq m(t)$, where $m(t)$ will be defined later (see Condition $(C2)$).

\begin{remark}
The measures $M^t_C,$ $N^{t}_{C,m},$ and $m^t_C$ are particular examples of measures considered in \cite{Rogers1970}. Following the terminology introduced in \cite{Rogers1970}, the set function $q(\cdot,t)$ satisfies the definition of the pre-measure \cite[Definition 5]{Rogers1970}, the measures $M^t_C$ and $N^{t}_{C,m},$ are obtained from this pre-measure by Method I \cite[Theorem 4]{Rogers1970}, while $m^t_C$ is obtained from $q(\cdot,t)$ by Method II \cite[Theorem 15]{Rogers1970}.
\end{remark}

The main properties of $M^t_C$ and $N^{t}_{C,m}$ are summarized below.

\begin{thm}\label{thm:properties-of-M}
The set functions $M^t_C$ and $N^{t}_{C,m}$  satisfy the following properties:
\begin{enumerate}
\item[$\mathrm{(1)}$] For all $t\in \mathbb{R}$:
\begin{enumerate}
\item[$\mathrm{(a)}$] If $M^t_C(\emptyset)=N^{t}_{C,m}(\emptyset)=0$, then $M^t_C$ and $N^{t}_{C,m}$ define outer measures;
\item[$\mathrm{(b)}$] $M^t_C(Z)  \leq m^t_C(Z)$ and $M^t_C(Z)  \leq N^{t}_{C,m}(Z)$ for every set $Z\subset X$;
\end{enumerate}
\item[$\mathrm{(2)}$] For $t>\mathrm{dim} _CZ$:
\begin{enumerate}
\item[$\mathrm{(a)}$] $M^t_C(Z) = m^t_C(Z)= 0$;
\end{enumerate}
\item[$\mathrm{(3)}$] For $t<\mathrm{dim} _CZ$: 
\begin{enumerate}
\item[$\mathrm{(a)}$] $  0 < M^t_C(Z)  \leq   N^{t}_{C,m}(Z) < \infty $; 
\end{enumerate}
\item[$\mathrm{(4)}$]  For $t= \mathrm{dim}_CZ$:
\begin{enumerate}
\item[$\mathrm{(a)}$]  If $0<m^t_C(Z)\leq\infty$, then $0 < M^t_C(Z)  \leq   N^{t}_{C,m}(Z) < \infty$;
\item[$\mathrm{(b)}$]  If $m^t_C(Z)=0$, then $ M^t_C(Z)=0$.
\end{enumerate}
\end{enumerate}
\end{thm}

\begin{remark}
Observe that in the case $(2)(a)$ and $(4)(b)$ one may have $ N^{t}_{C,m}(Z) >0$.
\end{remark}

\begin{proof}
Statement $(1)(a)$ follows from Theorem 4 in \cite{Rogers1970}. Statement $(1)(b)$ follows directly from the definition of  $M^t_C$ and $N^{t}_{C,m}$.
To show Statement $(2)(a)$ it is enough to observe that by Proposition \ref{prop:crit-value}, $m^t_C(Z)= 0$ for all $t>\mathrm{dim} _CZ$. Then by Statement $(1)(b)$ of the Theorem, $0 \leq M^t_C(Z)  \leq m^t_C(Z)=0$.

To prove Statement $(3)(a)$ first observe that the outer measures $M^t_C$ and $N^{t}_{C,m}$ are always finite. That is because the set function $q(\cdot,t)$ is finite and for every set $Z\subset X$ one can find a finite cover $\{C_i\}_{i=1}^{K}$ by cylinders of length $m$.
Then we have
$$  M^t_C(Z)\leq N^{t}_{C,m}(Z) \leq \sum_{i=1}^K q(C_i,t) < \infty.  $$ 

On the other hand, $M^t_C(Z)=0$ implies $m^t_C(Z)=0$.
To see this, observe that there are finitely many cylinders of a given length $s$. Since $0<q(C,t)<\infty$ for any cylinder $C$, there are positive real numbers $\{\gamma_s\}_{s\geq 1}$ such that $q(C,t)\geq \gamma_s>0$ for any cylinder of length $s$.
Assume that $M^t_C(Z)=0$. For any large $L\in\mathbb{N}$ choose $\epsilon>0$ such that $\epsilon< \min\{ \frac{1}{L},\gamma_1,\ldots,\gamma_L   \}$. There exists a cover $\{ C_i \}_{i\geq 1}$ of $Z$ by cylinders such that
$$   \sum_i q(C_i,t) < \epsilon. $$
By the choice of $\epsilon$, we must have that $l(C_i)>L$ for each $i\geq 1$. 
Letting $L\to \infty$ we obtain that $m^t_C(Z)=0$.

 By Proposition  \ref{prop:crit-value}, $m^t_C(Z)= \infty$ for all $t<\mathrm{dim} _CZ$. By (the contraposition of) the above argument, this implies that $M^t_C(Z)>0$. By Statement (1)(b), 
 this completes the proof of $(3)(a)$.

The proof of $(4)(a)$ is identical to the proof of $(3)(a)$.
Statement $(4)(b)$ is a direct consequence of $(1)(b)$.
\end{proof}

For simplicity of notation, if there is no confusion, we will simply write $\mathfrak{M}^t$, $m^t$, $M^t$ and $N^{t}_{m}$.

Our arguments require the $C$-structure to satisfy the following additional conditions. We remark that all the conditions stated below are naturally satisfied by 
 $C$-structures corresponding topological entropy, Hausdorff dimension, and topological pressure of H\"older continuous potentials, which we prove in Section \ref{s:app}.

\begin{enumerate}
\item[(C1)] There exists a uniform constant $Q_1>0$ such that for every $t< \dim_C(X)$ and for every cylinder $C$ there is a cylinder $\tilde{C}^t$ such that $C \subseteq \tilde{C}^t$ and one has:
$$ Q_1  q(\tilde{C}^t,t) \leq M^t (C) \leq \min\{ q(C,t),q(\tilde{C}^t,t)\}  ;$$
\item[(C2)]  For every $t< \dim_C(X)$ there exists $m=m(t) \in \mathbb{N}$ such that for every cylinder $C$ whose length is a multiple of $m$ there is a cylinder $\tilde{C}^t_m$ whose length is a multiple of $m$  such that $C \subseteq \tilde{C}^t_m$, and one has:
$$  N^t_{m} (C) = q(\tilde{C}^t_m,t) ; $$ 
\item[(C3)] There exists a uniform constant $Q_3>1$ such that for any two words $u, v$, such that $uv$ is an admissible word, and any $t\in \mathbb{R}$ one has
$$    Q_3^{-1}   q(C(u), t)  q(C( v), t) \leq q(C( u v), t)  \leq Q_3  q(C(u), t)  q(C( v), t);$$
\item[(C4)] For any two cylinders such that $A\subset B$ one has $\eta(A)\leq \eta(B)$.
\end{enumerate}

\section{Large intersection classes}\label{subsection:lic}

As  previously mentioned, Theorem \ref{thm:main} follows from a stronger result that the set $E$ in Theorem \ref{thm:main} belongs to a cartain class of sets, which are in some sense large.

We consider the following classes of sets, which are defined by generalizing classes introduced by F\"arm and Persson in \cite{FP2011} as  modifications of Falconer's intersection classes from \cite{Falconer1994}.
\begin{dfn}\label{def:class}
Let $\mathcal G^t (X)$, $ t < \dim_C(X)$ be the class of $G_{\delta}$-sets $F \subset X$ such that
\[
  M^t (F \cap C) = M^t (C)
\]
holds for all cylinders $C$.
\end{dfn}

Our main results on these classes are the following theorems.
The first theorem is a generalization of \cite[Theorem 1]{FP2011}.
\begin{thm}\label{thm:main1}
Assume that the Carath\'eodory structure satisfies Conditions $(C1)$ and  $(C4)$. 
Then the classes $\mathcal{G}^t (X)$ are closed under countable intersections and the Carath\'eodory dimension of any set in one of these classes is at least $t$.
\end{thm}

As an important application of  large intersection property, F\"arm and Persson  calculated the Hausdorff dimension of the intersection of irregular sets over countably many \emph{different dynamical systems}, see \cite[Proposition 1]{FP2013}.

We prove Theorem \ref{thm:main1} in Section \ref{s:1}. In the lemma below we observe that the claim about the Carath\'eodory dimension of sets in  $\mathcal{G}^t (X)$ is a natural consequence of Definition \ref{def:class}.

\begin{lem}\label{lem:dimension}
If $F\in \mathcal{G}^t (X),$ then $\dim_C(F)\geq t$.
\end{lem}

\begin{proof}
Since  $t < \dim_C(X)$, by Statement $(3)$ in Theorem \ref{thm:properties-of-M}, $M^t(X)>0$. 
Then there is a cylinder $C$ such that $M^t(C)>0$. By Definition \ref{def:class}, $M^t(F)>0$.
By Statement $(2)$ in Theorem \ref{thm:properties-of-M}, $\dim_C(F)\geq t$.
\end{proof}

Let $\mathcal A_x$ be the set of accumulation points of $\{ \delta _x^n \}_{n\geq 1}$.
For a sequence $\mathcal J=\{ \mu ^{(\ell )} \}_{\ell \in \mathbb N}$  of probability measures on $X$, 
we denote by $\Delta  (\mathcal J )$ the set of finite convex combinations of measures in $\mathcal{J}$. Namely, 
\[ 
 \Delta (\mathcal J ) = \bigcup _{L\geq 1} \Delta _L(\mathcal J ), \quad \Delta _L(\mathcal J )=  \left \{  \mu _{\mathbf t}  (\mathcal J ) \mid \mathbf t\in A_L  \right\} , 
\]
where 
$
A_L=\left\{ (t_0, t_1, \ldots ,t_L) \in [0,1] ^{L+1}  \mid  \sum _{\ell =0}^L t_\ell =1 \right\} 
$
and
$
 \mu _{\mathbf t} (\mathcal J ) = \sum _{\ell =0} ^L t_\ell  \mu ^{(\ell )}$
for $\mathbf t =(t_0, t_1, \ldots ,t_L)  \in A_L$.
We define the \emph{saturated set} $E (\mathcal J ) $ of $\mathcal J $ by
\[
E (\mathcal J ) = \left \{ x\in X \mid \Delta (\mathcal J ) \subset \mathcal A_x
\right\} 
\]
(cf.~\cite{PS2007,CZ2013}).
For a  probability measure $\mu$ on $X$, define
 the \emph{generic set} $G(\mu ) $
 by
\[
 G(\mu ) =\{ x\in X \mid \lim _{n\to \infty} \delta _x^n = \mu \} .
\]
%(These sets are called a \emph{saturated set} and a \emph{generic set} of $\mu$, respectively; cf.~\cite{PS2007}.)
The following theorem is a generalization of %F\"arm-Persson theorem 
\cite[Theorem 2]{FP2011} and is the central result of this paper.

\begin{thm}\label{cor:main}
Assume that the Carath\'eodory structure satisfies Conditions $(C1)-(C4)$. 
The following statements hold for any sequence $\mathcal J=\{ \mu ^{(\ell )} \}_{\ell \geq 0}$  of ergodic invariant probability measures:
\begin{enumerate}
\item[$\mathrm{(1)}$] $E (\mathcal J) \in \mathcal G^t$ for all $t<\inf \left\{ \mathrm{dim} _C(\mu^{(\ell)})  \mid \ell \geq 0 \right\}$;
\item[$\mathrm{(2)}$]  $    \mathrm{dim} _C(E (\mathcal J))   \geq     \inf \left\{ \mathrm{dim} _C(\mu^{(\ell)})  \mid \ell \geq 0 \right\}$.  
\end{enumerate}
\end{thm}

The proof of Theorem \ref{cor:main} occupies all of Section \ref{s:2}.

\begin{remark}
It is worth pointing out that Theorem \ref{thm:main1} holds for any shift space (not necessarily a subshift of finite type)
and any Carath\'eodory structure satisfying Conditions $(C1)$ and  $(C4)$. In the proof of Theorem \ref{cor:main} (Lemma \ref{lem:0218}, Lemma \ref{lem:0218a2}, and Lemma \ref{lem:1010})
we use the fact that the underlying shift space is a subshift of finite type. By slightly modifying the arguments one should be able to extend this result to any shift with specification.  
\end{remark}

\begin{remark}
We remark that (2) of Theorem \ref{cor:main} gives another proof of (a special version of) Theorem 1.2 of Chen-Zhou \cite{CZ2013} for  
  the  topological pressure 
  of saturated sets, 
  although their method is  different from ours.
  In particular, it is unclear whether the lower bound of (2) for  the  Hausdorff dimension  and the large intersection property in (1) follow from their approach.
\end{remark}

To conclude Theorem \ref{thm:main}, we also need the following proposition, which we prove on the next page.
\begin{prop}\label{thm:3}
Let  $\mathcal J=\{ \mu ^{(\ell )} \}_{\ell \geq 0}$ be a linearly independent sequence   of invariant 
   probability measures on $X$.
Then, $x\in E (\mathcal J ) $  implies that
\[
\lim _{\epsilon \to 0} \frac{\log \mathscr E_x(\epsilon )}{-\log \epsilon } =\infty .
\]
\end{prop}

The next theorem summarizes results in this section and gives the main result in this paper. 

\begin{thm}\label{thm:main-general}
Let $X $ be a topologically mixing subshift of finite type of $  \{ 1, 2, \ldots , \kappa\} ^{\mathbb N}$ with $\kappa\geq 2$.
 Let $f:X \to X$ be the left-shift operation on $X$.  
 Let $E$ be the set of points $x\in X$ satisfying
 \begin{equation*}
 \lim _{\epsilon \to 0}  \frac{\log \mathscr E_x(\epsilon ) }{-\log \epsilon } =\infty.
%  \quad \text{for all 
%   $x\in E$}.
\end{equation*}
Assume that the Carath\'eodory structure satisfies Conditions $(C1)-(C4)$ and the following condition:

\begin{enumerate}
\item[$\mathrm{(C5)}$] For any $t<\mathrm{dim} _C(X)$, there is a  linearly independent sequence   of invariant 
   probability measures $\{ \mu ^{(\ell )} \}_{\ell \geq 0}$ on $X$ such that $\mathrm{dim} _C(\mu ^{(\ell )}) > t$ for all $\ell \geq 0$.
\end{enumerate}

\noindent Then,  for any $t< \mathrm{dim} _C(X)$, there is an $E_t\subset E$ such that $E_t\in  \mathcal G^t$. In particular, 
$$    \mathrm{dim} _C(E) =  \mathrm{dim} _C(X).  $$ 
\end{thm}

\begin{proof}
 Theorem \ref{thm:main-general} immediately follows from Theorems \ref{thm:main1}, \ref{cor:main} and  Proposition \ref{thm:3}.
\end{proof}

\begin{proof}[Proof of Theorem \ref{thm:main}]
We will see in Section \ref{s:app} that $C$-structures corresponding topological entropy, Hausdorff dimension, and topological pressure satisfy Conditions $(C1)-(C4)$. 
%The fact that those structures also satisfy Condition $(C5)$ follows from \cite[Theorem 6.6]{BS2000}. 
Hence, Theorem \ref{thm:main} immediately follows from Theorem \ref{thm:main-general}, if we prove $(C5)$.
Let $\psi(x):X\to \mathbb{R}$ be any H\"older continuous function not cohomologous to a constant.
Observe that then for any two distinct values $t_1,~ t_2$, the corresponding potentials $\varphi + t_1\psi,~  \varphi+t_2\psi $
are not cohomologous and by Proposition 20.3.10 in \cite{Kat}  have distinct unique ergodic equilibrium measures $\mu_{t_1}$ and  $\mu_{t_2}$ respectively.
In addition, by \cite[Theorem 11.6]{pes97}, we have for any $t\in\mathbb{R}$

$$ |\mathrm{dim} _C(X)   -    \mathrm{dim} _C(\mu_t)|       =  | P(\varphi) -    P_{\mu_t}(\varphi)  |  =    | P(\varphi) -    \int  \varphi  d\mu_t - h_{\mu_t}(f)  | $$    
$$ \leq     | P(\varphi) -    \int  \varphi   + t\psi  d\mu_t - h_{\mu_t}(f)  |   +   |t|\|\psi\|         
= | P(\varphi) - P(\varphi + t\psi)   |    +   |t|\|\psi\|  \leq  2 |t|\|\psi\|, $$
where in the last inequality we used continuity of the pressure \cite[Theorem 11.4]{pes97} and $\| \cdot \|$ denotes the supremum norm in the space of continuous functions.

\end{proof}

\begin{proof}[Proof of Proposition \ref{thm:3}]
For a subset $A$ of $\mathcal P(X)$, let $N(\epsilon ,A)$ denote the $\epsilon$-covering number of $A$ with respect to the first Wasserstein metric $d$.
Then,  it is straightforward to see that 
 $\mathscr E_x(\epsilon )=N(\epsilon , \mathcal A_x)$, and recall that $\Delta (\mathcal J) \subset \mathcal A_x$ for each $x\in E(\mathcal J)$.
 We will  show that 
\begin{equation}\label{eq:0312a}
\lim _{\epsilon \to 0} \frac{\log N (\epsilon , \Delta( \mathcal J ) )}{-\log \epsilon} = \infty ,
\end{equation}
which implies the conclusion 
by  the above observations.

 Let $L\geq 1$ be an integer.
 Since $\mu ^{(0)}, \ldots , \mu ^{(L)}$ are linearly independent, $\Delta _L(\mathcal J)$ is an $L$-dimensional simplex.
 Therefore, it is easy to see that its box-counting dimension 
\[
\lim_{\epsilon \to 0} \frac{\log N (\epsilon , \Delta _L( \mathcal J ) )}{-\log \epsilon} 
\]
is well-defined and equal to $L$.
It follows from this and the observation $\Delta _L( \mathcal J )\subset \Delta ( \mathcal J )$ that
\[
\liminf _{\epsilon \to 0} \frac{\log N (\epsilon , \Delta ( \mathcal J ) )}{-\log \epsilon} \geq L .
\]
Since $L$ is arbitrary, this  implies \eqref{eq:0312a} and completes the proof.
\end{proof}

\section{Applications}\label{s:app}

\subsection{$C$-structure corresponding to topological pressure}\label{subsection:pressure}

In this section we fix a H\"older continuous function $\varphi: X \to \mathbb{R}$ and consider the following $C$-structure $\tau=(\mathcal{F},\xi,\eta,\psi)$ on $X$.  

Given a cylinder $C$, 
  we define
\begin{align*}
&\xi(C):= \exp\left(S_{l(C)}\varphi(C)\right) := \exp\left(\sup_{x\in C}\sum_{k=0}^{l(C)-1}\varphi(f^k(x))\right), \quad \\
&\eta(C):=e^{-l(C)}, \quad \psi(C):=\frac{1}{l(C)},
\end{align*}
and also set $\eta(\emptyset)=\psi(\emptyset)=\xi(\emptyset)=0$.
It is easy to see that the collection of subsets $\mathcal{F}$ and the functions $\xi,\eta,\psi$ satisfy Conditions (A1), (A2) and (A3), and hence, introduce the Carath\'eodory dimension structure on $X$.

The corresponding outer measures are given by 

$$ 
\mathfrak{M}^t_C(Z,1/n):=\inf  \left\{\sum_{i} e^{S_{l(C_i)}\varphi(C_i)  -l(C_i)t}   ~| ~   l(C_i) \geq n, Z \subset \bigcup_i C_i \right\},   
$$

$$ 
m^t_C(Z):= \lim_{n \to \infty } \inf \left\{\sum_{i} e^{S_{l(C_i)}\varphi(C_i)  -l(C_i)t}   ~| ~   l(C_i) \geq n, Z \subset \bigcup_i C_i \right\},   \text{ and}
$$

$$ 
M^t_C(Z):=\inf  \left\{\sum_{i} e^{S_{l(C_i)}\varphi(C_i)  -l(C_i)t}   ~| ~   Z \subset \bigcup_i C_i \right\}.
$$

The corresponding Carath\'eodory dimension of a set $Z$ is exactly the topological pressure $P(\varphi,Z)$ on $Z$ as defined by Pesin-Pitskel in \cite{PP84}. We will denote 
$P=P(\varphi):=P(\varphi,X)=\dim_C(X)$.
We also observe that topological entropy corresponds to $P(\varphi)$ for $\varphi = 0$.

By the lemma below, all the results in this paper apply to this structure. 

\begin{lem}\label{lem:prop-car} The above Carath\'eodory structure satisfies Conditions $(C1)-(C4)$.
\end{lem}

\begin{remark}
In the proof of Lemma \ref{lem:prop-car} (more precisely in verifying Conditions $(C1)$, $(C2)$, and $(C3)$) we used both the fact that $X$ is a subshift of finite type and that $\varphi$ is H\"older continuous.
Our arguments can also be extended to any shift with specification and a function $\varphi$ satisfying Bowen property (see \cite[Definition 20.2.5]{Kat}).
Deducing Theorem \ref{thm:main} in this case would also require verifying Condition $(C5)$ in Theorem \ref{thm:main-general}. 
\end{remark}

\begin{proof}
First observe that $M^t (C) \leq q(B,t)$ for any cylinders $C$ and $B$ with $C\subseteq B$. This follows from the definition of $M^t$. 
We now show that $M^t (C) \geq  Q q(\tilde{C}^t,t) $ for some cylinder $\tilde{C}^t$ containing $C$ with some positive number $Q$ being independent of $C$ and $t$.
Let $\{ C_i\}$  be a cover of C by cylinders of length greater than $\l(C)$. By H\"older continuity of $\varphi$, there exists $Q_0>0$ such that
\[
  \sum_{i} e^{S_{l(C_i)}\varphi(C_i)} \geq Q_0 e^{S_{l(C)}\varphi(C)}   \sum_{i} e^{S_{l(C_i)-l(C)}\varphi(f^{l(C)}C_i)}. 
  \]

If $C=C(u)$ for some admissible word $u=u_1 \ldots u_l$ and $M_{u_l a}=1$ for some $a\in \{ 1, \ldots, \kappa\},$ then
$f^{l(C)}(C)$ covers $C(a)$. In particular, the sets  $f^{l(C)}C_i$ form a cover of $C(a)$. Define $M_{min}:= \min_{a\in\{     1,\ldots, \kappa\}  } M^t(C(a)) >0    $. We have that

%Observe that the cylinders $f^{l(C)}C_i$ form a cover of $X$. We then have
\[
 \sum_{i} q(C_i,t) \geq Q_0    e^{S_{l(C)}\varphi(C)} e^{-l(C)t}  \sum_{i} q(   f^{l(C)}C_i, t) \geq Q_0 q(C,t) M_{min}.
\]
This proves $(C1)$.

For $(C2)$ let $\{C_i\}$ be a cover of a cylinder $C$ such that $l(C_i)\geq l(C) +m $ for some large $m$, to be determined later.
Let $\{A_j\}$ be a cover of $C$ by cylinders of length $l(C)+m$. Using similar reasoning as above, we have
\begin{multline*}
 \sum_{i} q(C_i,t)
 \geq Q_0^2    e^{S_{l(C)}\varphi(C)} e^{-l(C)t}  \\
 \times \sum_j e^{S_m\varphi(f^{l(C)}A_j)}e^{-mt} \sum_{C_i\subset A_j}     e^{S_{l(C_i)-l(C)-m}\varphi(f^{l(C)+m}C_i)} e^{-(      l(C_i)-l(C)-m )t}\\
\geq  Q_0^2 q(C,t) M_{min}  \sum_j e^{S_m\varphi(f^{l(C)}A_j)}e^{-mt}.
\end{multline*}
As a consequence of uniform counting estimates for the partition sums (see for example \cite[Proposition 20.3.2]{Kat}), there is $Q_1>0$ such that
$$     \sum_j e^{S_m\varphi(f^{l(C)}A_j)} \geq Q_1 e^{mP}.  $$
Together we obtain,
$$ \sum_{i} q(C_i,t)\geq    Q_0^2 q(C,t) M_{min} Q_1 e^{m(P-t)}.$$
Choosing $m$ large enough so that $    Q_0^2  M^t(X) Q_1 e^{m(P-t)} >1   $ completes the proof of $(C2)$.

Condition $(C3)$ is a direct consequence of H\"older continuity of $\varphi$,
and 
Condition $(C4)$ follows directly from the definition of $\eta$.
\end{proof}

\subsection{$C$-structure corresponding to Hausdorff dimension}
Results in this paper can be used in particular to study Hausdorff dimension of irregular sets. 
For this it is enough to set $\xi(C)=1$ for every cylinder $C$ and define $\eta(C)$ as the diameter of $C$ (and $\psi (C) =1/ l(C)$ as in the previous example).
It is automatic that such defined $C$-structure satisfies conditions $(C1)-(C4)$.

\section{Proof of Theorem \ref{thm:main1}}\label{s:1}

\begin{lem}
If $0<c\leq 1$ and $F$ is a set such that

$$  M^t(F\cap C ) \geq c M^t(C)  $$
holds for all cylinders $C$, then
$$  M^t(F\cap U) \geq c M^t(U)  $$
holds for open sets $U$.

\end{lem}

\begin{proof}
The proof is identical to the proof of Lemma 1 in \cite{FP2011} if we replace $M^t_{\infty}$ with $M^t$ and $d(\cdot)^t$ with $q(C,t)$. We provide the argument for clarity.

Let $U \subset X$ be open. Then we can write $U$ as a countable union $U=\bigcup_i C_i$ of pairwise disjoint cylinders.
Let $\{D_j\}\subset \mathcal{F}$ be a cover of $ F\cap U$. We can assume that this cover is disjoint.

Given $C_i$, if there are $D_j\subset C_i$, we may write
$$ \sum_{D_j\subset C_i} q(D_j,t) \geq M^t(F\cap C_i) \geq c M^t(C_i).  $$
Here we used the fact that two cylinders are either disjoint or one of them is contained in the other. Hence if  $D_j\subset C_i$ for some $C_i$, then all such sets $D_j$ form a cover of $C_i\cap F$. 

We can construct a disjoint cover $\{ \tilde{C}_k \}$ of $U$ by replacing each collection $C_{i_k}$ contained in some $D_j$ with the cylinder $D_j$.
We then obtain
\[
  \sum_j q(D_j,t) \geq c \sum_{k} M^t(\tilde{C}_k) \geq c M^t(U).
\]    
Taking the infimum over all covers $\{D_j\}$ finishes the proof.
\end{proof}

\begin{lem}\label{lem:cto1}
If $c>0$ and $F$ is a set such that
$$  M^{t_0}(F\cap C ) \geq c M^{t_0}(C)  $$
holds for some $t_0 < \dim_C(X)$ and all cylinders $C$, then
$$  M^t(F\cap C ) \geq  M^t(C)  $$
holds for all cylinders $C$, and $ t \leq t_0. $
\end{lem}

\begin{proof}
The proof that  $M^{t_0}(F\cap C ) \geq  M^{t_0}(C) $ is identical to the proof of Lemma 2 in \cite{FP2011} if we replace $M^t_{\infty}$ with $M^t$ and $d(\cdot)^t$ with $q(C,t)$. 
It uses Condition $(C1)$. Hence this part of the lemma holds for every Carath\'eodory structure satisfying $(C1)$. Here we provide the argument for clarity.

Let $\{C_i\}$ be a collection of cylinders covering $F\cap C$. We may assume that the cylinders are pairwise disjoint. Since $M^t(F\cap C)$ is finite, we may assume that 
$\sum_i q(C_i,t_0)$ is finite. Therefore, for every $\epsilon >0$, there exists $m_0$ such that
\begin{equation}\label{eqn:tail}
\sum_{\{ l(C_i) \geq m_0 \}  } q(C_i,t_0) < \epsilon.  
\end{equation}

Build a cover $\{  D_j  \}  $ of $C$ by cylinders of length $l(D_j)\leq m_0$ in the following way. Either $D_j=C_i$ for some $C_i$ with $l(C_i) < m_0$ or $C_i \cap F \subset D_j \cap F$ where $\l(C_i)\geq m_0$ and 
$  D_j\cap F \subset \bigcup_{\{ i |  C_i\cap F \subset D_j \cap F \}}      C_i \cap F$.
Observe that by the assumption of the lemma for every cylinder $D\subset C$ one has that $F\cap D \neq \emptyset$ so such cover exists.
In addition, in the latter case we can estimate
\[
  \sum_{  \{ i | C_i  \cap F \subset D_j \cap F  \}  } q(C_i,t_0) \geq M^{t_0} (F\cap D_j) \geq c M^{t_0}(D_j) \geq c Q_1 q(\tilde{D}^{t_0}_j,t_0)   
\]
by  $(C1)$.
It follows that
\[
   \sum_{  \{ i | l(C_i)< m_0   \}  } q(C_i,t_0) +  c^{-1}Q_1^{-1}\sum_{  \{ i | l(C_i) \geq m_0   \}  } q(C_i,t_0) \geq \sum_{j} q(\tilde{D}^{t_0}_j,t_0) .
   \]
Using this and (\ref{eqn:tail}) we obtain
\begin{multline*}
\sum_{  i  } q(C_i,t_0) = \sum_{  \{ i | l(C_i)< m_0   \}  } q(C_i,t_0) +  c^{-1}Q_1^{-1}\sum_{  \{ i | l(C_i) \geq m_0   \}  } q(C_i,t_0) \\
 + (1-c^{-1}Q_1^{-1})\sum_{  \{ i | l(C_i) \geq m_0   \}  } q(C_i,t_0)
\geq M^{t_0}(C) + (1-c^{-1}Q_1^{-1})\epsilon.
\end{multline*}
Letting $\epsilon \to 0$ and taking infimum over all covers $\{C_i\}$ we conclude
\[
  M^{t_0}(F\cap C) \geq M^{t_0}(C).  
\]

We now turn to the case $t< t_0$. 
Let $\{C_i\}$ be a collection of cylinders covering $F\cap C$. We may assume that each $C_i$ is contained in $C$ so that $\eta(C_i) \leq \eta(C)$ by Condition $(C4)$. 
We then have 
\begin{multline*}
  \sum_{  i  } q(C_i,t) = \sum_{  i  } \xi(C_i)\eta(C_i)^t= \sum_{  i  } \xi(C_i)\eta(C_i)^{t_0}\eta(C_i)^{t-t_0}\geq \eta(C)^{t-t_0}\sum_{  i  } q(C_i,t_0)\\
  \geq \eta(C)^{t-t_0}M^{t_0}(C\cap F)\geq
 \eta(C)^{t-t_0}M^{t_0}(C)\geq
Q_1   \eta(C)^{t-t_0}   q(\tilde{C}^{t_0},t_0) 
\end{multline*}
by $(C1)$. Using $(C4)$, we continue
\[
 \geq  Q_1   \eta(\tilde{C}^{t_0})^{t-t_0}  \xi(\tilde{C}^{t_0}) \eta(\tilde{C}^{t_0})^{t_0} = Q_1 q(\tilde{C}^{t_0},t) 
\geq Q_1 M^t(C),
\]
where the last inequality follows from the definition of $M^t$ after observing that $\tilde{C}^{t_0}$ covers $C$.

Taking infimum over all covers $\{C_i\}$ we conclude
\[
  M^{t}(F\cap C) \geq Q_1 M^{t}(C)  
  \]
and by the first part of the lemma, $  M^{t}(F\cap C) \geq  M^{t}(C)  $.
\end{proof}

The proof of Theorem \ref{thm:main1} is now a consequence of Lemma \ref{lem:dimension} and the following lemma.

\begin{lem}
If $F_i\in \mathcal{G}^t (X)$ for all $i\in \mathbb{N}$, then

$$   M^t\left(   \bigcap_i F_i\cap U   \right) = M^t(U)$$
for all open sets $U$, and $    \bigcap_i F_i \in \mathcal{G}^t (X). $
 
\end{lem}

\begin{proof}
The proof is very similar to the proof of Proposition 2 in \cite{FP2011}, which uses the increasing sets lemma for the outer measures \cite[Theorem 52]{Rogers1970}.

Let first $\{  F_i \}_{i=1}^{\infty}$ be a countable collection of open sets, with the property that
\[
  M^t(F_i\cap U) \geq M^t(U)  
\]   
holds for any open set $U$. 
Fix an open set $U\subset X,$ $\epsilon >0,$ and a sequence $\{ \epsilon_k\}_{k=0}^{\infty}$ of positive real numbers such that $\sum_{k=0}^{\infty}\epsilon_k < \epsilon$.

The main idea is to approximate $\bigcap_i F_i\cap U$ by a countable intersection of compact nested sets $\bigcap_i \bar{D}_i$ such that $\bar{D}_k\subset  \bigcap_{i=1}^{k}  F_i \cap U$ for $k\geq 1$.

We define the collection $\{D_i\}_{i=0}^{\infty}$ inductively in the following way.
Let $D_0$ be an open subset of $U$ with the property that $\bar{D}_0\subset U$ and that
\[
  M^t(D_0)>M^t(U) - \epsilon_0.  
\]
The fact that such a set exists follows from the increasing sets lemma \cite[Theorem 52]{Rogers1970}.
To see that this result applies to our setting it is enough to observe that the set function $q(\cdot,t)$ satisfies the definition of the pre-measure \cite[Definition 5]{Rogers1970}
and that $M^t(\cdot)$ is constructed from this pre-measure by Method I defined in \cite[Theorem 4]{Rogers1970}.

Having defined $D_i$ for $i=0,\ldots,k-1$, we choose $D_k$ to be an open subset of $F_k\cap D_{k-1}$ (note that $F_k\cap D_{k-1}\subset \bigcap_{i=1}^{k}F_i\cap U$ is an open set)
with the property that $\bar{D}_k\subset F_k\cap D_{k-1}$ and such that
\[
 M^t(D_k) > M^t(F_k\cap D_{k-1}) - \epsilon_k.  
\]
Observe that since each $D_k$ is an open set, %and each $F_k\in \mathcal{G}^t(X)$ is open 
we have that
\begin{multline*} 
 M^t(D_k) > M^t(F_k\cap D_{k-1}) - \epsilon_k \geq  M^t( D_{k-1}) - \epsilon_k\\
 >  M^t(F_{k-1}\cap D_{k-2}) - \epsilon_{k-1} - \epsilon_k > \ldots > M^t(U) - \epsilon. 
 \end{multline*} 
In addition, $\bigcap_{i=0}^{\infty}\bar{D}_i \subset \bigcap_{i-1}^{\infty}F_i\cap U$. Let $\{C_l\}$ be a cover of $\bigcap_{i=0}^{\infty}\bar{D}_i$ by cylinders.
Since $\{\bar{D}_k\}$ is a nested sequence of compact sets and $\bigcup_l C_l$ is open, there is $m\in\mathbb{N}$ such that $\bar{D}_m\subset\bigcup_l C_l$. Therefore
$$  \sum_{l}q(C_l,t)\geq M^t(\bar{D}_m)\geq M^t(U) - \epsilon.  $$
Letting $\epsilon \to 0$ we obtain
\[
  M^t(\bigcap_{i=1}^{\infty}F_i\cap U)\geq M^t(U).
\]
Therefore we have shown that any countable intersection of open sets in $\mathcal{G}^t(X)$ is in  $\mathcal{G}^t(X)$. The proof is finished by observing that any countable intersection of $G_{\delta}$ sets can be expressed
as a countable intersection of open sets.
\end{proof}

\section{Proof of Theorem \ref{cor:main}}\label{s:2}

In this section we fix an invariant ergodic probability measure $\mu$, $t < \mathrm{dim} _C(G(\mu ))$, and $m\geq m(t)$, where $m(t)$ is defined in Condition $(C2)$.
We then consider the corresponding outer measure $N^t_m,$ which is defined in Subection \ref{sec:modification}.

We say that a word $x$ is a subword of $y$, and write $x\preceq y$, if $C(y)\subset C(x)$. In addition, we say that $x$ is a proper subword of $y$, and write $x\prec y,$ if the inclusion is strict.
 Furthermore, we denote the concatenation of words  $x$ and $y$  by $xy$.

For $\mu \in \mathcal P(X)$, $n\in \mathbb N$ and $\epsilon >0$, 
we define $E(\mu , n, \epsilon )$ by
\[
E(\mu , n, \epsilon ) = \left\{ x\in X \mid \delta _x^n \in B_\epsilon (\mu ) \right\} ,
\]
where $B_\epsilon (\mu )$ is the ball of radius $\epsilon $ and with center $\mu$.

\subsection{Preliminary lemmas}

We start this subsection with a simple but crucial observation.
Namely, we note that there are finitely many words of length not exceeding $m$. Because of that there exists a constant $\alpha=\alpha(m,t)>0$ such that
\begin{equation}\label{eqn:alpha}
\alpha < q(C(u),t) < \alpha^{-1}
\end{equation}
for any word $u$ with $|u|\leq m$.

As a consequence of (\ref{eqn:alpha}) and Condition $(C3)$ we obtain the following.

\begin{lem}\label{lem:equivalence}
The outer measures $M^t$ and $N^t_m$ are equivalent.
\end{lem}

\begin{proof}
Clearly $M^t\leq N^t_m$. On the other hand, for any collection $\{  x^i \}_i$ of words of lengths $l_i=k_i m + n_i,$ with $k_i\in\mathbb{N}$ and $0\leq n_i\leq m,$ consider
the corresponding collection $\{ \underbar{$x$}^i \}_i$ of words of lengths $k_i m$ obtained by removing the last $n_i$ letters from the word $x^i$. Then
$$  \sum_i  q(C (\underbar{$x$}^i),t )  \leq   Q_3  \sum_i  \frac{q(C(x^i),t)}{q(C(   x^i_{k_im+1}\ldots x^i_{k_im + n_i}  ),t)}  \leq \alpha^{-1} Q_3  \sum_i  q(C(x^i),t).$$
In addition, $C(x^i)\subset C(\underbar{$x$}^i)$. Consequently, $N^t_m \leq   \alpha^{-1} Q_3 M^t. $ 
\end{proof}

The next two lemmas correspond to Lemma 4 and 6 of \cite{FP2011} and will be used as inductive steps in the following subsection.

\begin{lem}\label{lem:0217}

Assume there are numbers $c,\epsilon>0$ and a word $y$ of length $m$ satisfying
\[
N^t_m    \left( C(y) \cap E(\mu , N, \epsilon )\right ) <c q(  C(y) , t)  
\]
for some $N>4m/\epsilon$. Then for every word $z$ satisfying $|z|=km\leq  \epsilon N/4$, and such that $zy$ is an admissible word, one has
\[
N^t_m    \left( C(zy) \cap E(\mu , N+|z| , \epsilon/2 )\right ) <c Q_3^2 q(  C(zy) , t).  
\]

\end{lem}

Before we state the next lemma, we recall an important property of subshifts of finite type. There exists $\tau\in\mathbb{N}$ such that for any two admissible words $u$ and $v$ (of any length) there exists a word $\omega$
of length $\tau m$ such that the word $u\omega v$ is admissible. For the rest of this paper we denote by $\tau$ the positive integer with this property.

\begin{lem}\label{lem:0218}
Assume there are numbers $c,\epsilon>0$ and a word $y$ of length $m$ satisfying
\[
N_m ^{t}\left( C(y) \cap E(\mu , N, \epsilon )\right ) >c q( C(y),t) ,
\]
for some $N> 2m/\epsilon$.  Then for every word $z$ satisfying $|z|=km,$ $(k+\tau)m\leq  \epsilon N/2,$ one has
\[
N_m^{t}\left( C(z) \cap E(\mu , N+\tau m +|z|, 2\epsilon )\right ) >    Q^{-(\tau+1)}_3  \alpha^{\tau +1} c    N_m^{t}\left( C(z) \right ),
\]
where  $\alpha$ is as in (\ref{eqn:alpha}).
\end{lem}

In order to show Lemma \ref{lem:0217} and \ref{lem:0218},
we will use the following.
\begin{lem}\label{lem:0218b}
If $z$ is a word of length $n_1$ and $x\in X$ satisfies
$
zx \in E\left(\mu , n+n_1, \epsilon  \right) $,
then 
\[
x\in E\left( \mu , n, \epsilon +  \frac{2n_1}{n}\right).
\]
Conversely, if  $z$ is a word of length $n_1$, $x\in E(\mu , n, \epsilon )$, and $zx$ is an admissible word, then
\[
zx \in E\left(\mu , n+n_1, \epsilon + \frac{2n_1}{n} \right) .
\]
\end{lem}
\begin{proof}
For each $\varphi \in \mathrm{Lip } ^1(X, [-1, 1])$,
\begin{multline*}
\left\vert \int _X \varphi d\delta _x^n -  \int _X \varphi d\delta _{zx}^{n+n_1} \right\vert \\ 
\leq 
\left\vert \left(\frac{1}{n} -\frac{1}{n+n_1} \right) \sum _{j=0}^{n+n_1-1}\varphi \circ f^j (zx)\right\vert 
+
\left\vert \frac{1}{n} \sum _{j=0}^{n_1-1}\varphi \circ f^j (zx)\ \right\vert 
\leq  \frac{2n_1}{n} .
\end{multline*}
So, we have $d(\delta _x^n , \delta _{zx}^{n+n_1} ) \leq 2n_1/n$ by the Kantorovich-Rubinstein dual representation  of the first Wasserstein metric (cf.~\cite[(2.1)]{KNS2019}),
which immediately implies the conclusion.
\end{proof}

\begin{proof}[Proof of Lemma \ref{lem:0217}]
By assumption, there are words $(x_i)_{i=1}^I$ of lengths being multiples of $m$ and such that 
\begin{equation}\label{eq:0217d}
C(y) \cap E(\mu , N, \epsilon ) \subset \bigcup _{i=1}^I C(x^i)
\end{equation}
and
\[
\sum _{i=1}^I q(C(x^{i}),t)  < c q(C(y),t).
\]
From \eqref{eq:0217d} and Lemma \ref{lem:0218b}, it follows that
\[
C(zy) \cap E\left( \mu , N+ |z| , \epsilon - 2|z|/N \right) \subset \bigcup _{i=1}^I C(zx^i).
\]

On the other hand, by Condition $(C3)$,  we have	
\[
q(C(zx^i),t)  \leq  Q_3 q(C(z),t) q(C(x^i),t)   \leq  Q^2_3   \frac{q(C(zy),t)}{q(C(y),t)}    q(C(x^i),t)   .
\]
Observing that the length of each word $zx^i$ is a multiple of $m$, we can conclude 
\begin{multline*}
N^t_m \left(C(zy) \cap E\left( \mu , n+ |z| , \epsilon -  2|z|/N \right)\right) \leq \sum _{i=1}^I q(C(zx^i),t)  \\
 = Q^2_3   \frac{q(C(zy),t)}{q(C(y),t)}   \sum _{i=1}^Iq(C(x^i),t)   < Q^2_3 c q(C(zy)) .
\end{multline*}
Recall that by the assumption, $|z|  \leq \epsilon N/4$, so that $ \epsilon -  2|z|/N \geq \epsilon/2$. The conclusion follows.
\end{proof}

\begin{proof}[Proof of Lemma \ref{lem:0218}]
Let $u$ be a word of length $\tau m$ such that the word $zuy$ is admissible.
We consider all possible covers of the set $C(zuy) \cap E(\mu , N+\tau m+ |z|, 2 \epsilon)$ by cylinders of lengths of multiples of $m$. There are three possibilities:

\begin{enumerate}
\item $C(zuy) \cap E(\mu , N+\tau m+ |z|, 2 \epsilon)\subset C(w)$, where:
\begin{enumerate}
\item $ w\preceq z$;
\item $z \prec w\preceq zu$, that is, $w=z\bar{u}$ with $\bar{u}\preceq u$;
\end{enumerate}
\item there is a collection $(x_i)_{i=1}^I$ of words of lengths being multiples of $m$ and such that 
$C(zuy) \cap E(\mu , N+\tau m+|z|, 2 \epsilon ) \subset \bigcup _{i=1}^I C(zuyx^i).$
\end{enumerate}

In case $(1)(a)$ we necessarily have that $C(z)\subset C(w)$.
Consequently, 
\begin{equation}\label{eqn:omega}
q(C(w),t) \geq N^t_m(C(z)).  
\end{equation}

We now turn to case $(1)(b)$. Dividing $\bar{u}$ into subwords of lengths $m$ and applying (\ref{eqn:alpha}) to each segment, by Condition $(C3)$, we have that

$$q(C(\bar{u}),t) \geq Q^{-(\tau -1)}_3  \alpha^{\tau}.$$

Consequently, by Condition $(C3)$,

\begin{equation}\label{eqn:omegau}
q(C(w),t)   = q(C(z\bar{u}),t)   \geq  Q^{-1}_3    q(C(z),t)q(C(\bar{u}),t)     \geq   Q^{- \tau }_3  \alpha^{\tau}    N^t_m(C(z)).  
\end{equation}

We now consider case (2). Let $(x_i)_{i=1}^I$ be a collection of words of lengths being multiples of $m$ and such that 
\[
C(zuy) \cap E(\mu , N+\tau m+|z|, 2 \epsilon ) \subset \bigcup _{i=1}^I C(zuyx^i).
\]
Observe that 
\[
C(y) \cap E(\mu , N , \epsilon ) \subset \bigcup _{i=1}^I C(yx^i) .
\]
By Condition $(C3)$, it follows that,
\begin{align*}
\sum_i q(C(zuyx^i),t)&\geq Q^{-2}_3 q(C(z),t)q(C(u),t)\sum_i q(C(yx^i),t)\\
&=  Q^{-2}_3 q(C(z),t) q(C(u),t) q(C(y),t)\sum_i  \frac{q(C(yx^i),t)}{q(C(y),t)}.
\end{align*}

We now estimate some of the terms on the line above. Dividing $u$ into subwords of lengths $m$ and applying (\ref{eqn:alpha}) to each segment, by Condition $(C3)$, we have that

\begin{align*}
&q(C(u),t) \geq Q^{-(\tau -1)}_3  \alpha^{\tau}. \text{ In addition, } q(C(y),t)>\alpha \text{ by (\ref{eqn:alpha}).}\\
&\text{Since }\{ C(yx^i) \}_{i=1}^{I}\text{ covers }C(y) \cap E(\mu , N , \epsilon ), \text{we also have that }     \\
&\sum_i  \frac{q(C(yx^i),t)}{q(C(y),t)} \geq \frac{N^t_m(  C(y) \cap E(\mu , N , \epsilon )     )}{q(C(y),t)} >c \text{ by the assumption of the lemma.}  
\end{align*}

We conclude that
\begin{equation}\label{eqn:zyxi}
\sum_i q(C(zuyx^i),t)\geq Q^{-(\tau+1)}_3  \alpha^{\tau +1} c q(C(z),t) \geq  Q^{-(\tau+1)}_3  \alpha^{\tau +1} c  N^t_m(C(z)).
\end{equation}

Inequalities (\ref{eqn:omega}), (\ref{eqn:omegau}), and (\ref{eqn:zyxi}) together give that
\begin{multline*}
N^t_m(C(z) \cap E(\mu , N+\tau m+|z|, 2 \epsilon ))\geq N^t_m(C(zuy) \cap E(\mu , N+\tau m+|z|, 2 \epsilon )) \\
\geq  Q^{-(\tau+1)}_3  \alpha^{\tau +1} c  N^t_m(C(z)),
\end{multline*}
which completes the proof.
\end{proof}

\subsection{Key lemma}\label{sec:end}

The following is the most important lemma in the proof of Theorem \ref{cor:main}.
\begin{lem}\label{lem:0218a2}
 For each $t<\mathrm{dim} _C(G(\mu ))$ and $m\geq m(t)$, there is a constant $\hat{c}>0$ such that for any word $z$ and any $\epsilon >0$ one has
\begin{equation}\label{eq:0218d}
\liminf_{n\to\infty} N^{t}_m (C(z) \cap E(\mu , n ,\epsilon )) \geq \hat{c} N^{t}_m (C(z)) .
\end{equation}
In addition, $\hat{c}\geq \frac{\alpha^{\tau + 2}   Q_3^{-(\tau + 4)}}{4}$. 
\end{lem}

\begin{proof}
Fix $\epsilon>0$. We first show that (\ref{eq:0218d}) holds for words whose lengths are multiples of $m$.
Arguing by contradiction, we assume that there is a word $z$ of length $|z|=k m$, for some $k\in\mathbb{N}$ and a sequence $\{ a_n  \}\subset \mathbb{N}$ increasing to infinity such that
\[
N^{t}_m (C(z) \cap E(\mu , a_n + |z|  + \tau m ,4\epsilon )) < \frac{\alpha^{\tau+1}  Q_3^{-(\tau+3) }  }{4  }N^{t}_m (C(z)) 
\]
for all $n\in\mathbb{N}$, where $\alpha$ is the constant given in Lemma \ref{lem:0218}.
Applying Lemma \ref{lem:0218} with $c=\frac{1}{4Q_3^2}$, we obtain that for each word $y$ of length $m$,
\[
N^{t}_m (C(y) \cap E(\mu , a_n  ,2\epsilon )) \leq\frac{1}{4  Q^2_3 }    q(C(y),t)< \frac{1}{2  Q^2_3 }    q(C(y),t)
\]
for all $n\in\mathbb{N}$ such that $a_n>2(|z|+\tau m)/\epsilon$.
By this estimate and Lemma \ref{lem:0217}, for any word $w$ of length $|w|=jm$ with  $j\in \mathbb N$ we have 
\begin{equation}\label{eq:0218f4}
N^{t}_m (C(w) \cap E(\mu , a_n+|w| -m  ,\epsilon )) \leq\frac{1}{2}   q(C(w),t) 
\end{equation}
for all $n\in\mathbb{N}$ such that $a_n  \geq 4|w|/\epsilon $.

Fix a large number $N\in\mathbb{N}$.
We first apply \eqref{eq:0218f4} to $w=z$. Choose $n_0\in\mathbb{N}$ such that $a_{n_0}> \max \{ N,   4|z|/\epsilon \}$ and denote $b_0:=a_{n_0}+|z| -m$.
 One can find a finite cover $(C_i)_{i=1}^I$ of $C(z) \cap E(\mu , b_0 ,\epsilon )$ by cylinders of lengths $l(C_i)=l_i m$ such that
\[
\sum _{i=1}^I q(C_i,t) \leq \frac{2}{3} q (C(z),t).
\]

For each $1\leq i\leq I$, we again apply \eqref{eq:0218f4} to $C=C_i$. Choose $n_i\in\mathbb{N}$ such that $a_{n_i}> \max\{ N,  4l_im/\epsilon \}$ and denote $b_i:=a_{n_i} + (l_i-1)m$.
There is a finite cover $(C_{i,j})_{j=1}^{J_i}$ of $C_i \cap E(\mu , b_i ,\epsilon )$ by cylinders of lengths $l(C_{i,j})=l_{i,j}m$ such that
\[
\sum _{j=1}^{J_i} q(C_{i,j},t) \leq \frac{2}{3} q(C_i,t).
\]
Together we obtain that $\bigcup_{i=1}^{I}\bigcup_{j=1}^{J_i} C_{i,j}$ is a cover of $ C(z)  \cap  \bigcap_{i=0}^{I}  E(\mu , b_i  ,\epsilon )$ and
\[
 \sum_{i=1}^I   \sum _{j=1}^{J_i} q(C_{i,j},t) \leq   \frac{2}{3}    \sum_{i=1}^I q(C_i,t) \leq    \left( \frac{2}{3}  \right)^2 q (C(z),t) .
\]

Repeating this argument, we obtain for each $L\in \mathbb N$ a sequence $ \{  \tilde{b}_i  \}_{i=1}^{B(L)}$ of numbers $\tilde{b}_i \geq N$  and a cover $\{  \tilde{C}_j \}_{j=1}^{M(L)}$ of   $ C(z) \cap   \bigcap_{i=0}^{B(L)}  E(\mu , \tilde{b}_i  ,\epsilon )$
by cylinders of lengths being multiples of $m$ such that
\[
    \sum _{j=1}^{M(L)} q(\tilde{C}_{j},t) \leq      \left( \frac{2}{3}  \right)^L q (C(z),t) .
\]
In addition, for each $L\in \mathbb N$ we have that 
\[
\bigcap _{n\geq N}E(\mu , n,  \epsilon )    \subset    \bigcap_{i=0}^{B(L)}  E(\mu , \tilde{b}_i  ,\epsilon ).
\]
Therefore, $N^t_m  \left(   C(z)  \cap    \bigcap _{n\geq N}E(\mu , n,  \epsilon )    \right)=0. $

Since $N$ was arbitrary, by observing that  
\begin{equation}\label{eq:0731b}
 G(\mu ) =   \bigcap _{\widetilde \epsilon >0}\bigcup _{N\in \mathbb N}\bigcap _{n\geq N}E(\mu , n, \widetilde \epsilon ),
\end{equation}
we conclude that $N^{t}_m ( C(z) \cap G(\mu ) ) =0.$ We claim that  $N^{t}_m (G(\mu ) ) =0.$ 

Indeed, since $N^{t}_m ( C(z) \cap G(\mu ) ) =0,$ then for every small $\hat{\epsilon}>0$ there exists a cover $\{ C( \omega^{\hat{\epsilon}}_j )  \}$
of $C(z) \cap G(\mu )$ with
$$  \sum_{j}q(\omega^{\hat{\epsilon}}_j,t)< \hat{\epsilon}.   $$
By the proof of Theorem \ref{thm:properties-of-M}, there is $L=L(\hat{\epsilon})>0$ such that each word $ \omega^{\hat{\epsilon}}_j$ has length at least $L$.
In particular, choosing $\hat{\epsilon}>0$ small enough (thus ensuring that $L$ is large enough), we can ensure that each word $\omega^{\hat{\epsilon}}_j$ is of the form $zu_jv_j$  where $|u_j|=\tau m$.

Let $x= (x_1 x_2 \ldots)$ be a point in $G(\mu )$. One can find a word $u$ of length $\tau m$ such that $x'=(zux_1 x_2 \ldots ) \in X$. In fact, $ x' \in C(z) \cap G(\mu )$.
Consequently, $G(\mu ) \subset f^{|z|+\tau m}(C(z)   \cap G(\mu ))$.
In addition, taking the collection of cylinders $\{  D_j\} = \{ C(z  u_j v_j) \}$ covering  $C(z) \cap G(\mu )$, the corresponding collection $\{  D'_j\} = \{ C(v_j) \}  $ covers   $G(\mu )$.
By Condition $(C3)$, 
$$   q ( D'_j, t  ) = q(C(v_j),t)  \leq Q_3  \frac{  q(C(z u_j v_j),t)  }{q(C(z  u_j),t)} \leq   \frac{Q_3^{\tau +k}}{\alpha^{\tau + k}}     q(C(z u_j v_j),t) <   \frac{Q_3^{\tau +k}}{\alpha^{\tau + k}}  \hat{\epsilon}.$$ 
Letting $\hat{\epsilon} \to 0,$ we obtain that $N^{t}_m (G(\mu ) ) =0.$
By Theorem \ref{thm:properties-of-M}, we must have that $\mathrm{dim} _C(G(\mu )) \leq  t$. This contradicts with the assumption of the lemma.
We therefore proved that  for any word whose length is a multiple of $m$ the inequality (\ref{eq:0218d}) holds  with 
\[
\hat{c}_0:= \frac{\alpha^{\tau + 1}   Q_3^{-(\tau + 3)}}{4}. 
\]

Let now $x$ be a word of length $|x|=km+l$ where $k\in\mathbb{N}$ and $0<l<m$. Let $\underbar{$x$},\bar{x}$ be words of lengths $km$ and $(k+1)m$ respectively, such that $ \underbar{$x$}\prec x \prec \bar{x}$.

We first observe that, for any $n\in\mathbb{N},$ $C(\bar{x})\cap E(\mu , n ,\epsilon ) \subset C(x)\cap E(\mu , n ,\epsilon ),$ so that
$$   \liminf_{n\to \infty}  N^t_m (C(x)\cap E(\mu , n ,\epsilon ) ) \geq  \liminf_{n\to \infty}  N^t_m (C(\bar{x})\cap E(\mu , n ,\epsilon ) ) $$ 
 $$\geq  \hat{c}_0 N^t_m (C(\bar{x})) = \hat{c}_0 q(C(x'),t),$$
where $x' \preceq \bar{x}$ and $|x'|=k'm$ for some $k'\leq k+1$. If $x'=\bar{x},$ then by $(C3)$ and (\ref{eqn:alpha}),

$$ q(C(x'),t) \geq Q_3^{-1} q(C(\underbar{$x$}), t) q (C( \bar{x}_{km+1}\ldots \bar{x}_{(k+1)m}  ),t)  $$  
 $$\geq Q_3^{-1} \alpha   q(C(\underbar{$x$}), t) \geq Q_3^{-1} \alpha N^t_m (C(x)).   $$

If now $x'\prec \bar{x}$, then we must have that $C(x)\subset C(x')$. Consequently, $$  q(C(x'),t)\geq N^t_m(C(x)). $$
We conclude that 

$$   \liminf_{n\to \infty}  N^t_m (C(x)\cap E(\mu , n ,\epsilon ) ) \geq \hat{c}_0 Q_3^{-1}\alpha      N^t_m (C(x))  =    \frac{\alpha^{\tau + 2}   Q_3^{-(\tau + 4)}     }{4} N^t_m (C(x))    , $$
which completes the proof.
\end{proof}

\subsection{Proof of Theorem \ref{cor:main}}\label{sec:proof-of-cor}

Let $\bar{t}<  \inf \left\{ \mathrm{dim} _C(\mu^{(\ell)})  \mid \ell \geq 0 \right\}$. Observe that  for $\ell\geq 0$, $ \mathrm{dim}_C(G(\mu^{(\ell)} )) \geq \mathrm{dim} _C(\mu^{(\ell)})$ by Birkhoff's ergodic theorem.
Therefore, Lemma \ref{lem:0218a2} holds for $\bar{t}$ and each $\mu^{(\ell)}$.
Fix $m \geq m(\bar t)$.

Recall that $[x]_n$ is the truncation $[x]_n=(x_1x_2 \cdots x_n)$ of $x=(x_1x_2 \cdots)$ and 
$d_{X}(x ,y ) = \sum _{j=1}^\infty \vert x_j -y_j \vert \beta ^{-j}$ for $x=(x_1 ,x_2, \ldots )$, $y=(y_1, y_2, \ldots ) \in X$.
Note that  
if $x\in X,$ $n\in\mathbb{N}$, $y\in C([x]_{n+T_1 m})$ and  $\varphi \in \mathrm{Lip } ^1(X, [-1, 1])$, then
\[
 \left\vert \int _X \varphi d\delta _x^n -  \int _X \varphi d\delta _{y}^{n} \right\vert   
\leq \frac{1}{n}  \sum _{j=0}^{n-1}   d_X(f^j (x), f^j (y)) \leq
\frac{1}{n(\beta - 1) }  \sum _{j=0}^{n-1}  \beta^{-(n+T_1m  -j)} .
\]
Let $T_1$ be  a positive integer such that $1/(n(\beta - 1) )  \sum _{j=0}^{n-1}  \beta^{-(n+T_1m  -j)} 
<\epsilon /4$ for any $n\in \mathbb N$.
By the Kantorovich-Rubinstein dual representation  of the first Wasserstein metric, we have $d(\delta _x^n , \delta _y^n ) \leq \epsilon /4$ for any $x\in X,$ $n\in\mathbb{N}$ and $y\in C([x]_{n+T_1 m})$.

Recall that $\tau$ denotes the specification constant, that is, for any two admissible words $u$ and $v$ (of any length) there exists a word $\omega$
of length $\tau m$ such that the word $u\omega v$ is admissible. Finally, we denote $T= (T_1+\tau)m$.

\begin{lem}\label{lem:1010}
For any $k\in \mathbb N$, $\{ \ell _j \}_{j=1}^k\subset \mathbb N$, $\lambda_1, \lambda _2 , \ldots , \lambda_k \in (0, 1]$ with $\sum_{j=1}^{k}\lambda_j =1$, $\epsilon >0$ and $m_0\in \mathbb N$, we have that
\[
\bigcup _{n\geq m_0} E\Big(   \sum_{j=1}^k  \lambda_j \mu^{(\ell _j)} , n ,\epsilon \Big) \in \mathcal G^{\bar t}(X).
\]
\end{lem}
\begin{proof}
Fix  $k\in \mathbb N$, $\{ \ell _j \}_{j=1}^k\subset \mathbb N$, $\lambda_1, \lambda _2 , \ldots , \lambda_k \in (0, 1]$ such that $\sum_{j=1}^{k}\lambda_j =1$, $\epsilon >0$ and $m_0\in \mathbb N$.
Fix also
a cylinder $C(z)\subset X$.
It follows from Lemma \ref{lem:0218a2} that we can find  a positive integer $\widetilde n\geq m_0$ such that:
\begin{itemize}
\item $|z| < \min_{j} \{   \lambda_j \widetilde n  \}$;
 \item $T k   / \widetilde n<\epsilon/4;$ 
\item  $N^{\bar{t}}_m (C(z) \cap E(\mu^{(\ell _j)} , n ,\epsilon/4 )) \geq c_1 N^{\bar{t}}_m (C(z)) $ for each $j=1,\ldots k$ and $n\geq   \lambda_j  \widetilde n$ with some $c_1>0$.
\end{itemize}
Denote $$  n_j = \lfloor \lambda_j \widetilde n \rfloor , \quad N_j =Tj  +\sum _{i=1}^j n_i  \quad \text{ for }j=1,\ldots   ,k-1, \quad \text{and} \quad  N= N_{k-1} + n_k$$
with $N_0=0$.
We consider the sets $ E:= C(z) \cap   E(   \sum_{j=1}^k  \lambda_j \mu^{(\ell _j)} , N ,\epsilon ),$
and
\[
 E_0:= C(z) \cap  \Big( \bigcap _{j=1}^k f^{-N_{j-1}} \Big(E(      \mu^{(\ell _j)} , n_j ,\epsilon/2 ) \Big)\Big) .
\]

Observe
 that for each $x\in X$,
\[
\delta_x^N = 
\frac{\widetilde n}{N}
   \sum _{j=1}^k 
   \frac{n_j}{\widetilde n} 
   \delta^{n_j} 
_{f^{N_{j-1}}(x)} ,
\]
and thus 
\[
d\Big(\delta _x^N ,  \sum_{j=1}^k  \lambda_j \mu^{(j)} \Big)
 \leq \left\vert 1- \frac{\tilde n}{N}\right \vert 
+ \sum _{j=1}^N\left\vert \lambda_j   - \frac{n_j}{ \widetilde n}\right\vert + \sum _{j=1}^k \lambda _j d\Big(  \delta^{n_j} 
_{f^{N_{j-1}}(x)} ,\mu ^{(\ell _j)}\Big).
\]
This implies that $E_0\subset E$.

By the choice of $T_1$, for $j=1,\ldots, k$ we have that, if
 $y^j
\in  E(      \mu^{(\ell _j)} , n_j ,\epsilon/4 ),$ then $C([y^j] 
_{n_j+T_1m})\subset  E(      \mu^{(\ell _j)} , n_j ,\epsilon/2 )$.
In addition, if $y^1, \ldots, y^k \in X$ are such that 
\[
y^1\in C(z) \cap  E(      \mu^{(1)} , n_1 ,\epsilon/4 ) \text{  and }  y^j \in    E(      \mu^{(\ell _j)} , n_j ,\epsilon/4 ) \text{ for } j=2,\ldots k, 
\]
then there are words $u^1,\ldots, u^{k-1}$ such that the length of each $u^k$ is $\tau m$ and  that the word
\[
w:=  [ y^1]_{n_1+T_1m}u^1 [y^2 ]_{n_2+T_1m}u^2  \cdots     [y^{k-1}]_{n_{k-1}+T_1m} u^{k-1} [y^k]_{n_k+T_1m} 
 \]
 is admissible. Denote by $\mathcal{W}$ the collection of all words $\omega$ obtained this way. We then have that
$
C(  w)
   \subset E_0$ for every $\omega\in \mathcal{W}.$

From this and the definition of $N^{\bar{t}}_m$ it follows that there is a positive number $c_2$ such that 
\begin{multline*} 
N^{\bar{t}}_m(E) \geq N^{\bar{t}}_m(E_0)    \geq       N^{\bar{t}}_m(  \bigcup_{\omega\in\mathcal{W}}    C(\omega)  )      \\
\geq  c_2  N^{\bar{t}}_m(   C(z) \cap  E(      \mu^{(\ell _1)} , n_1 ,\epsilon/4 )   )    N^{\bar{t}}_m(   E(      \mu^{(\ell _2)} , n_2 ,\epsilon/4 )   )   \ldots   N^{\bar{t}}_m(   E(      \mu^{(\ell _k)} , n_k ,\epsilon/4 )   ) \\   
\geq     c_2  c_1^k   N^{\bar{t}}_m(C(z)) .
\end{multline*}
Since $N\geq m_0$, this completes the proof.
\end{proof}

Let $\{   \nu^{(\ell)}    \} _{\ell \in \mathbb N}$ be a countable dense subset of $\Delta  (\mathcal J )$, then
it is straightforward to see that $E (\mathcal J) = \bigcap _{\ell \in \mathbb N} E(\nu^{(\ell)} )$.
For each $\ell \in \mathbb N$, one can find $k\in \mathbb N$ 
and $\lambda_1, \lambda _2 , \ldots , \lambda_k \in [0, 1]$ such that $\sum_{j=1}^{k}\lambda_j =1$ and $\nu ^{(\ell )} = \sum _{j=1}^k \lambda _j\mu^{(j)}$.
It follows from Lemma \ref{lem:1010}
 %and Theorem \ref{thm:main1}
  that
\[
E\Big(      \sum_{j=1}^k  \lambda_j \mu^{(j)}  \Big )  = \bigcap_{m_1\in\mathbb{N}}  \bigcap_{m_0\in\mathbb{N}}   \bigcup_{n\geq m_0}   E\Big(  \sum_{j=1}^k  \lambda_j \mu^{(j)}  , n ,m_1^{-1} \Big) \in    \mathcal{G}^{\bar t}(X).
 \]
Therefore, Theorem \ref{cor:main}  immediately follows from Theorem \ref{thm:main1}.

\section*{Acknowledgments}
We are sincerely grateful to Tomas Persson for many fruitful discussions and for helping us better understand the ideas in his paper, \cite{FP2011}.
This work has been communicated with Dmitry Dolgopyat and Yakov Pesin, we would like to thank them for their valuable suggestions.
We also would like to express our  gratitude to Yong Moo Chung,  Naoya Sumi and Kenichiro Yamamoto for many valuable comments. 
This work was partially supported by JSPS KAKENHI
Grant Numbers 19K14575 and 19K21834.

\end{document}